\documentclass[final]{siamltex1213}
\title{CONVERGENCE OF A CELL-CENTERED FINITE VOLUME DISCRETIZATION
FOR LINEAR ELASTICITY}

\author{Jan Martin Nordbotten\footnotemark[2]}
\slugger{sinum}{xxxx}{xx}{x}{x--x}

\usepackage[tbtags]{amsmath}
\usepackage{amssymb}

\DeclareMathOperator{\card}{card}
\DeclareMathOperator{\tr}{tr}
\DeclareMathOperator{\argmin}{arg\,min}
\def\abs#1{\lvert #1\rvert}
\def\norm#1{\lVert #1\rVert}
\def\dblbr#1{[[#1]]}
\def\aver#1{\langle #1\rangle}

\newcommand*{\om}{\boldsymbol{\omega}}
\newcommand*{\sig}{\boldsymbol{\sigma}}
\def\underkern#1#2#3{
        {} \mskip#1mu \underline{\mskip-#1mu #3 \mskip-#2mu} \mskip#2mu {}
}
\def\overkern#1#2#3{
        {} \mskip#1mu \overline{\mskip-#1mu #3 \mskip-#2mu} \mskip#2mu {}
}
\newcommand*{\unabla}{\underkern11\nabla}
\newcommand*{\bnabla}{\overkern11\nabla}
\newcommand*{\ubnabla}{\underkern11{\overkern11\nabla}}

\newcommand*{\ab}{\boldsymbol{a}}
\newcommand*{\fb}{\boldsymbol{f}}
\newcommand*{\gb}{\boldsymbol{g}}
\newcommand*{\nb}{\boldsymbol{n}}
\newcommand*{\ub}{\boldsymbol{u}}
\newcommand*{\vb}{\boldsymbol{v}}
\newcommand*{\wb}{\boldsymbol{w}}
\newcommand*{\xb}{\boldsymbol{x}}
\newcommand*{\yb}{\boldsymbol{y}}
\newcommand*{\phib}{\boldsymbol{\varphi}}

\newcommand*{\Hb}{\boldsymbol{H}}
\newcommand*{\Ib}{\boldsymbol{I}}
\newcommand*{\Ob}{\boldsymbol{0}}
\newcommand*{\Pb}{\boldsymbol{P}}
\newcommand*{\Sb}{\boldsymbol{S}}
\newcommand*{\Tb}{\boldsymbol{T}}
\newcommand*{\Ub}{\boldsymbol{U}}

\newcommand*{\Cc}{\mathcal{C}}
\newcommand*{\Dc}{\mathcal{D}}
\newcommand*{\Fc}{\mathcal{F}}
\newcommand*{\Gc}{\mathcal{G}}
\newcommand*{\Hc}{\mathcal{H}}
\newcommand*{\Pc}{\mathcal{P}}
\newcommand*{\Sc}{\mathcal{S}}
\newcommand*{\Tc}{\mathcal{T}}
\newcommand*{\Vc}{\mathcal{V}}

\newcommand*{\Hbc}{\boldsymbol{\mathcal{H}}}

\newcommand*{\Bfr}{\mathfrak{B}}
\newcommand*{\Dfr}{\mathfrak{D}}
\newcommand*{\Ffr}{\mathfrak{F}}
\newcommand*{\Ifr}{\mathfrak{T}}
\newcommand*{\Rfr}{\mathfrak{R}}

\newcommand*{\mcell}{\mathfrak{k}}
\newcommand*{\medge}{\mathfrak{z}}
\newcommand*{\mvert}{\mathfrak{s}}
\newcommand*{\bmcell}{\mathfrak{l}}

\newtheorem{assumption}[theorem]{Assumption}
\newtheorem{definitions}[theorem]{Definitions}

\begin{document}
\maketitle

\renewcommand{\thefootnote}{\fnsymbol{footnote}}

\footnotetext[2]{Department of Mathematics,
University of Bergen and
Department of Civil and Environmental Engineering,
Princeton University.}

\renewcommand{\thefootnote}{\arabic{footnote}}

\begin{abstract}
We show convergence of a cell-centered finite volume discretization for
linear elasticity. The discretization, termed the MPSA method, was recently
proposed in the context of geological applications, where cell-centered
variables are often preferred. Our analysis utilizes a hybrid variational
formulation, which has previously been used to analyze finite volume
discretizations for the scalar diffusion equation. The current analysis
deviates significantly from previous in three respects. First, additional
stabilization leads to a more complex saddle-point problem. Secondly,
a discrete Korn's inequality has to be established for the global
discretization. Finally, robustness with respect to the Poisson ratio is
analyzed. The stability and convergence results presented herein provide
the first rigorous justification of the applicability of cell-centered
finite volume methods to problems in linear elasticity.
\end{abstract}

\begin{keywords}\end{keywords}

\begin{AMS}\end{AMS}

\pagestyle{myheadings}
\thispagestyle{plain}
\markboth{JAN MARTIN NORDBOTTEN}{CONVERGENCE OF FINITE VOLUMES FOR LINEAR ELASTICITY}

\section{Introduction}\label{sec1}

We consider the following problem of isotropic (but heterogeneous) linear
elasticity \cite{kik88}
\begin{equation}\label{eq1.1}
\begin{aligned}
&\nabla\cdot\sig+\fb=\Ob && \text{in } \Omega\\ 
&\sig = 2\mu\unabla\ub+\lambda (\tr \unabla \ub)\Ib && \text{in } \Omega\\
&\ub = \gb_D&& \text{on } \Gamma_D\\
&\sig\cdot\nb=\gb_N&& \text{on } \Gamma_N\\
\end{aligned}
\end{equation}
Here the domain $\Omega$ is a bounded connected polygonal subset of $\mathbb{R}^d$, with
boundary $\partial\Omega=\Gamma_D\cup\Gamma_N$. We have introduced the symmetric gradient
operator by the notion $\unabla \ub\equiv(\nabla\ub+\nabla\ub^T )/2$. Furthermore,
let the parameter functions $\fb\in(L^2)^d$ and the Lam\'e parameters
$0<\underline{\mu}\le\mu(\xb)\le\bar{\mu}$ and $\lambda$ be bounded and positive, defined almost
everywhere. If $\Gamma_D$ has positive measure, equations~\eqref{eq1.1} have a unique
weak solution in $(H^1 (\Omega))^d$. Otherwise, if $\Gamma_N=\partial\Omega$, equations~\eqref{eq1.1}
have a unique weak solution in $(H^1 (\Omega))^d/\Rfr(\Omega)$, where $\Rfr$ is
the space of rigid body motions: $\Rfr(\Omega)=\{\ab+\nobreak\om\wedge\nobreak \xb;\ \ab,\om\in\mathbb{R}^d\}$. In the
latter case $\int_\Omega\fb\,d\xb=\int_{\partial\Omega}\gb_N\,d\Sb$ is a necessary
compatibility condition on the data. Without loss of generality,
we will assume, by subtracting any smooth function satisfying the
boundary conditions and correspondingly modifying the right-hand side,
that both $\gb_D=\Ob$ and $\gb_N=\Ob$. We note in particular that we do not consider
transformation which is available for the (simpler) case of homogeneous
Dirichlet boundary conditions, when equations~\eqref{eq1.1} can be recast in a
locally coercive form \cite{kik88}.

In the continuation, it will be convenient to refer directly to the
weak form of equations~\eqref{eq1.1}. We will here, and in the following use the
convention $\Hb^1 (\Omega)\equiv(H^1 (\Omega))^d$, and tacitly assume that the space is
restricted to the homogeneous Dirichlet boundary condition. The weak form
of equations~\eqref{eq1.1} then takes the form (see e.g.~\cite{kik88}): Find $\ub\in\Hb^1 (\Omega)$
such that
\begin{equation}\label{eq1.2}
\int_\Omega2\mu\unabla\ub:\unabla\vb 
+\lambda(\nabla\cdot\ub)(\nabla\cdot\vb)\,d\xb=\int_\Omega\fb\cdot\vb\,d\xb 
\quad \text{for all }
\vb\in\Hb^1 (\Omega) 
\end{equation}

Recently, we proposed to extend cell-centered finite volume methods
for the scalar diffusion equation to analyze finite volume methods for
elasticity. Cell-centered finite volume methods may be advantageous
for problems associated with poro-elastic materials in geological
applications. In particular, it is advantageous to A) Exploit a
similar grid and data structure between the fluid flow and mechanical
discretizations \cite{lem14,nor14}, B) Share the same restrictions on non-matching
grids and hanging nodes for both the flow and mechanical discretizations,
and C) Have explicit force balance and traction at surfaces and grid
faces (often associated with fractures and faults). We are therefore in
particular interested in the generalization of the so-called Multi-Point
Flux Approximation (MPFA)-O method for scalar equations \cite{aav02}. This variant
of finite volume method has proved popular in applications and is also
amenable to theoretical analysis. Most notably, $2^\text{nd}$-order convergence
in potential and $1^\text{st}$-order convergence in flux was established using a
link to mixed finite element methods already in \cite{kla06,whe06}, while analysis of
the method following the discrete functional framework was presented in
\cite{age08}. The generalization of the MPFA-O method to linear elasticity was
formulated for general anisotropic and heterogeneous problems in \cite{nor14b}. It
is there termed the Multi-Point Stress Approximation (MPSA) method,
and was supported by extensive numerical experiments indicating robust
convergence results for a wide range of grids and Poisson ratios. It is the
objective of this paper to provide a theoretical convergence analysis of
this method by generalizing the discrete functional framework for finite
volume methods.

The discrete functional framework for finite volume methods is
detailed in \cite{eym06}. This approach was utilized by \cite{eym07} to develop finite
volume discretizations for the scalar diffusion equation for which
convergence could be proved under quite weak assumptions on the grid and
coefficients. Furthermore, the framework was adapted to non-symmetrical
discretizations in \cite{age08} and \cite{age10} to generalize and prove convergence of
the MPFA methods.

The main obstacle in order to extend the analysis of discretizations
for scalar diffusion equations to discretizations of equations~\eqref{eq1.1} is
to ensure coercivity of the discretization. In particular, the discrete
functional spaces previously used for the scalar problem are conceptually
similar to the Crouzeix--Raviart finite element space. This space does not
satisfy Korn's inequality and thus lacks stability. Our work therefore
extends the spaces to allow for a natural stabilization analogous to
discontinuous Galerkin methods \cite{han02}. Furthermore, to account for the
additional challenges associated with the lack of local coercivity for
equation~\eqref{eq1.1}, we will additionally need to lean on ideas from variational
multiscale \cite{hug98} and discontinuous Galerkin \cite{bre03} methods in our analysis. We
will also address the issue of stability with respect to the Poisson ratio
(so-called numerical locking) by reverting to ideas from mixed methods
\cite{bre91,lem14}.

We note previous work on finite volume methods for elasticity. Most work
where node-centered \cite{bai95} and cell-centered \cite{wen03} finite volume methods
are introduced for elasticity contain only numerical validation. This
includes also recent work on cell-centered methods \cite{nor14b,car14}. When additional
variables are introduced, convergence of finite volume methods has been
established \cite{mar11}. Similarly, convergence has recently been established for
face-valued finite volume methods \cite{lem14}. This latter citation is particularly
important, as the method is furthermore shown to be locking-free. To
the knowledge of this author, this contribution represents the first
rigorous convergence proof for cell-centered finite volume methods on
general grids and heterogeneities. Furthermore, we establish that the
method is locking free for a large class of grids.

The manuscript is structured as follows. In section~\ref{sec2}, we establish
notation and recall the formulation and main results from the
hybrid-variational framework. In section~\ref{sec3}, we define our discrete
mixed variational problem and establish its connection to the MPSA-O
discretization. In section~\ref{sec4}, we establish a local coercivity condition
which guarantees the global coercivity of the discretization. This
represents the key technical obstacle in order to extend previous work
and prove stability of the discretization. In section~\ref{sec5}, we largely
exploit previous work on discrete functional discretizations to obtain
the main convergence result. Section~\ref{sec6} provides detailed comments on the
method, including application to homogeneous Dirichlet problems, reduced
integration on simplex grids, and the corresponding scalar diffusion
discretization. Section~\ref{sec7} details how the local coercivity condition
simplifies to an explicit condition on the mesh for homogeneous problems,
and addresses the issue of numerical locking. Section~\ref{sec8} concludes the
paper.

\section{Discrete functional framework}\label{sec2}
In this section we give the definition of our finite volume mesh and
discrete variables.

\subsection{Finite volume mesh}\label{sec2.1}
Following \cite{age10}, we modify the construction of \cite{eym06}, and denote a finite
volume mesh by the triplet $\Dc=(\Tc,\Fc,\Vc)$, representing the mesh Tessellation,
Faces, and Vertexes, such that:
\begin{itemize}
\item $\Tc$ is a non-overlapping partition of the domain $\Omega$. Furthermore,
let $m_K$ denote the $d$-dimensional measure of $K\in\Tc$.

\item $\Fc$ is a set of faces of the partitioning $\Tc$. We consider only cases
where elements $\sigma\in\Fc$ are subsets of $d-1$ dimensional hyper-planes
of $\mathbb{R}^d$, and all elements $\sigma\in\Fc$ we associate the $d-1$ dimensional
measure $m_\sigma$. Naturally, the faces must be compatible with the
mesh, such that for all $K\in\Tc$ there exists a subset $\Fc_K\subset\Fc$ such
that $\partial K=\bigcup_{\sigma\in\Fc_K}\sigma$.

\item $\Vc$ is a set of vertexes of the partitioning $\Tc$. Thus for any $d$
faces $\sigma_i\in\Fc$, either their intersection is empty or $\bigcap_i\sigma_i
=s\in\Vc$.
\end{itemize}

Note that in the above (and throughout the manuscript), we abuse notation
by referring to the object and the index by the same notation. E.g.,
we will by $K\in\Tc$ allow $K$ to denote the index, as in $\Fc_K$, but also the
actual subdomain of $\Omega$, such that the expression $\partial K$ is meaningful.

Additionally, we state the following useful subsets of the mesh triplet,
which allows us to efficiently sum over neighboring cells, faces or
vertexes:
\begin{itemize}
\item For each cell $K\in\Tc$, we denote the faces that comprise its boundary
by $\Fc_K$ and the vertexes of $K$ by $\Vc_K$. We will associate with each
corner $s\in\Vc_K$ a subcell of $K$, identified by $(K,s)$, with a volume
$m_K^s$ such that $\sum_{s\in\Vc_K}m_K^s =\nobreak m_K$.

\item For each face $\sigma\in\Fc$, we denote the neighboring cells $\Tc_\sigma$ and
its corners for $\Vc_\sigma$. Note that for all internal faces $\Tc_\sigma$ will
contain exactly two elements, while it contains a single element
when $\sigma\subset\partial\Omega$. We will associate with each corner $s\in\Vc_\sigma$ a
subface of $\sigma$, identified by $(s,\sigma)$, with an area $m_\sigma^s$ such
that $\sum_{s\in\Vc_\sigma}m_\sigma^s =m_\sigma$.

\item For each vertex $s\in\Vc$, we denote the adjacent cells by $\Tc_s$ and
the adjacent faces by $\Fc_s$.
\end{itemize}

We associate for each element $K\in\Tc$ a unique point (cell center) $\xb_K\in K$
such that $K$ is star-shaped with respect to $\xb_K$, and we denote the
diameter of $K$ by $d_K$. Furthermore, we denote the distance between cell
centers $\xb_K$ and $\xb_L$ as $d_{K,L}=\abs{\xb_K-\xb_L}$. The grid diameter is denoted
$h=\max_{K\in\Tc}d_K$.

We associate with each face $\sigma$ its outward normal vector with respect
to the cell $K\in\Tc_\sigma$ as $\nb_{K,\sigma}$, and the Euclidian distance to the cell
center $d_{K,\sigma}$. For each subface $(s,\sigma)$ we denote the subface center as
$\xb_\sigma^s$ and the smallest set of Gauss quadrature points sufficient for
exact integration of second-order polynomials on $(s,\sigma)$ as $\Gc_\sigma^s$. For
each quadrature point $\beta\in\Gc_\sigma^s$ we associate the position $\xb_\beta$ and
weight $\omega_\beta$.

We associate for each vertex $s\in\Vc$ its coordinate $\xb_s\in\Omega$.

The above definition covers all 2D grids of interest. However, we place
two restrictions on grids in 3D: Firstly, the above definition of a mesh
requires all cell faces to be planar. The analysis that follows can be
extended, at the cost of extra notation, to the case of non-planar faces
can be allowed for as long as the faces allow for a piece-wise planar
subdivision associated with the face corners $\Vc_\sigma$. Secondly, we will
require that no more than three faces meet at a vertex. This permits
quadrilaterals, simplexes, and all so-called $2.5$-D grids (e.g.\ 2D
horizontal grids extended vertically, as common in the petroleum
industry). However, as an example we do not permit certain 3D grids
such as pyramids. The formulation of the method readily generalizes to
this case, however the application of Korn's inequality (section~\ref{sec4.3})
becomes more technical.

Regularity assumptions on the discretization $\Dc$ are detailed elsewhere
(see e.g.\ \cite{eym06}), we will in the interest of simplicity of exposition
henceforth assume that the classical grid regularity parameters (grid
skewness, internal cell angles, and coordination number of vertexes)
do not deteriorate.

\subsection{Discrete variables and norms}\label{sec2.2}

In contrast to MPFA methods for the scalar diffusion equation \cite{aav02}, it is
for the MPSA O-method not sufficient to use so-called continuity points to
provide sufficient constraints to yield a unique discretization \cite{nor14b}. Thus,
we need to extend the discrete function spaces utilized previously \cite{age10,eym07}
to allow for multiple unknowns per subface. We detail the three discrete
spaces used in our analysis below.

The following discrete space is classical \cite{eym06}:

\begin{definition}\label{def2.1}
For the mesh $\Tc$, let $\Hc_\Tc (\Omega)\subset L^2 (\Omega)$ be the set of
piece-wise constant functions on the cells of the mesh $\Tc$.
\end{definition}

As with the dual interpretation of the elements $K\in\Tc$, the space $\Hc_\Tc
(\Omega)$ is isomorphic to the space of discrete variables associated with the
cell-center points $\xb_K$. There should also be no cause for confusion in the
following when we work with the vector-valued spaces, still denoted $\Hc_\Tc$.

For the space $\Hc_\Tc$ we introduce the inner product
\[
[u,v]_\Tc=\sum_{K\in\Tc}\sum_{\sigma\in\Fc_K}\frac{m_\sigma}{d_{K,\sigma}} 
(\gamma_\sigma u-u_K)
(\gamma_\sigma v-v_K) 
\]
and its induced semi-norm
\[
\abs{u}_\Tc=([u,u]_\Tc )^{1/2}
\]
Here the operator $\gamma_\sigma u$ interpolates the piecewise constant values of
$\Hc_\Tc$ onto the faces of the mesh, weighted by the distances $d_{K,\sigma}$.
\[
\gamma_\sigma u=\biggl(\sum_{K\in\Tc_\sigma}\frac{u_K}{d_{K,\sigma}}\biggr)/
\biggl(\sum_{K\in\Tc_\sigma}d_{K,\sigma}^{-1}\biggr) 
\quad\text{for all }\sigma\in\Fc;\
\sigma\notin\Gamma_D
\]

For Dirichlet boundary edges, $\sigma\in\Gamma_D$, we take $\gamma_\sigma u=0$. Equivalently,
the operator $\gamma_\sigma u$ can be defined as the value which minimizes the
definition of the semi-norm $\abs{u}_\Tc$. We note that this semi-norm, and
those that follow, are equivalent to full norms when $\Gamma_D$ has positive
measure. In the case where $\Gamma_N=\partial\Omega$, rigid body motions must be excluded,
as we will note later.

The following space is the discontinuous discrete space which we use
to construct the consistent gradient functions of our scheme. It is to
our knowledge novel in the context of analysis of finite volume methods,
however it is natural in the sense that it is a discrete version of the
first-order discontinuous Galerkin space (as we will emphasize later).

\begin{definition}\label{def2.2}
For the mesh triplet $\Dc$, let $\Hc_\Dc$ be the set of real scalars
$(u_K,u_{K,s}^{\sigma,\beta})$, for all $K\in\Tc$, for all $(s,\sigma)\in\Vc_K\times\Fc_K$ and for
all $\beta\in\Gc_\sigma^s$.
\end{definition}

The space $\Hc_\Dc$ thus contains one unknown per cell, in addition to multiple
unknowns on each interior sub-face. This will be essential to control
the space $\Rfr(\Omega)$. As above, we will immediately take $u_{K,s}^{\sigma,\beta}=0$
for all $\sigma\in\Gamma_D$.

We denote for all internal subfaces
$\dblbr{u}_s^{\sigma,\beta}=u_{R,s}^{\sigma,\beta}-u_{L,s}^{\sigma,\beta}$ for $u\in\Hc_\Dc$ and
$\Tc_\sigma=\{R,L\}$ as the jump in the discrete function $u$ across that edge. We will
also need a notion of an average face value, and we denote similarly for
all internal subfaces $\aver{u}_s^\sigma=\frac1{m_s^\sigma}\sum_{\beta\in\Gc_s^\sigma}\omega_\beta
\frac{u_{R,s}^{\sigma,\beta}+u_{L,s}^{\sigma,\beta}}2$. For boundary edges $\sigma\in\partial\Omega$ only
one function value is available and we define $\dblbr{u}_s^{\sigma,\beta}=0$ and
$\aver{u}_s^\sigma=\frac1{m_s^\sigma}\sum_{\beta\in\Gc_s^\sigma}\omega_\beta u_{R,s}^{\sigma,\beta}$.
We now associate with the space $\Hc_\Dc$ the inner product (note that unless
explicitly marked with parenthesis, summation lasts the full equation):
\[
[u,v]_\Dc=\sum_{K\in\Tc}\sum_{s\in\Vc_K}\sum_{\sigma\in\Fc_s}\frac{m_K^s}{d_{K,\sigma}^2}
(u_K-\aver{u}_s^\sigma)(v_K-\aver{v}_s^\sigma)+\frac{m_K^s}{d_{K,\sigma}^2}
\frac1{m_s^\sigma}\sum_{\beta\in\Gc_s^\sigma}\omega_\beta \dblbr{u}_s^{\sigma,\beta} \dblbr{v}_s^{\sigma,\beta}
\] 
and the induced semi-norm
\[
\abs{u}_\Dc=([u,u]_\Dc)^{1/2}
\]
As with $\abs{u}_\Tc$, it is straight-forward to also define a proper norm
for $\Hc_\Dc$.

The above ``discontinuous'' discrete space generalizes the
``continuous'' discrete space, which we recall as \cite{age10}:

\begin{definition}\label{def2.3}
For the mesh triplet $\Dc$, let $\Hc_\Cc$ be the set of real scalars
$(u_K,u_s^\sigma )$, for all $K\in\Tc$ and for all $(s,\sigma)\in\Vc_K\times\Fc_K$.
\end{definition}

By introducing the natural projection operator $\Pi_D\colon\Hc_\Cc\to\Hc_\Dc$ as $(\Pi_D
u)_K=u_K$; $(\Pi_D u)_{K,s}^{\sigma,\beta}=u_s^\sigma$ for all $K\in\Tc$ and for all
$(s,\sigma)\in\Vc_K\times\Fc_K$, we can immediately define the inner product
\[
[u,v]_\Cc=[\Pi_D u,\Pi_D v]_\Dc
\]
and the induced semi-norm
\[
\abs{u}_\Cc=([u,u]_\Cc)^{1/2}
\]

In addition to the projection operator defined above, we shall need a
few more operators to move between function spaces.
\begin{itemize} 
\item Let the operator $\Pi_\Tc\colon\Hc_\Dc\to\Hc_\Tc$ be defined as $(\Pi_\Tc u)(x)=u_K$ for
all $x\in K$ and $K\in\Tc$. Furthermore, as there should be no reason for
confusion we also define $\Pi_\Tc\colon\Hc_\Cc\to\Hc_\Tc$ with as $(\Pi_\Tc u)(x)=(\Pi_\Tc
\Pi_\Dc u)(x)=u_K$ for all $x\in K$ and $K\in\Tc$. Finally, we also write
$\Pi_\Tc\colon C(\Omega)\to\Hc_\Tc$ as $(\Pi_\Tc u)(x)=u(x_K)$ for all $x\in K$ and $K\in\Tc$.

\item Let the operator $\Pi_\Cc\colon\Hc_\Dc\to\Hc_\Cc$ be defined as $(\Pi_\Cc u)_K=u_K$; $(\Pi_\Cc
u)_s^\sigma=\aver{u}_s^\sigma$ for all $K\in\Tc$ and for all $(s,\sigma)\in\Vc_K\times\Fc_K$.
\end{itemize}

The spaces defined above satisfy the following inequalities.
\begin{itemize}
\item Discrete Sobolev inequality \cite{eym06}: For all $u\in\Hc_\Tc$ and for
all $q\in[1,2d/(d-\nobreak 2+\nobreak \epsilon))$
\[
\norm{u}_{L^q}\le q\,C_{\mathit{sob}} \abs{u}_\Tc
\]

\item Relationship between $\Hc_\Tc$ and $\Hc_\Cc$ \cite{age10}: For all $u\in\Hc_\Cc$
\[
\abs{\Pi_\Tc u}_\Tc\le\sqrt{d}\, \abs{u}_\Cc
\]

\item Relationship between $\Hc_\Cc$ and $\Hc_\Dc$ (trivial from definitions):
For all $u\in\Hc_\Dc$
\[
\abs{\Pi_\Cc u}_\Cc\le\abs{u}_\Dc
\]
\end{itemize}

Finally, we introduce local spaces $\Hc_{\Dc,s}\subset\Hc_\Dc$ for each $s\in\Vc$ defined
such that $u\in\Hc_{\Dc,s}$ if $u_{K,t}^{\sigma,\beta}=0$ for all $t\in\Vc$ with $t\neq s$ and
$u_K=0$ if $s\notin\Vc_K$. Similarly, $\Hc_{\Tc,s}$ and $\Hc_{\Cc,s}$ are defined through the
projection operators defined above. The local spaces have the natural
norms, which to be precise are given for all $u\in\Hc_\Dc$ as
\[
\abs{u}_{\Dc,s}^2=\sum_{K\in\Tc_s}\sum_{\sigma\in\Fc_s}\frac{m_K^s}{d_{K,\sigma}^2}
(u_K-\aver{u}_s^\sigma)^2+\frac{m_K^s}{d_{K,\sigma}^2}\frac1{m_s^\sigma}
\sum_{\beta\in\Gc_s^\sigma}\omega_\beta (\dblbr{u}_s^{\sigma,\beta})^2 
\]
And for all $u\in\Hc_\Tc$ as
\[
\abs{u}_{\Tc,s}^2=\sum_{K\in\Tc_s}\sum_{\sigma\in\Fc_s\bigcap\Fc_K}\frac{m_\sigma^s}{d_{K,\sigma}}
(\gamma_\sigma u-u_K)^2 
\]
Such that both
\[
\abs{u}_\Dc^2=\sum_{s\in\Vc}\abs{u}_{\Dc,s}^2\quad\text{and}\quad
\abs{u}_\Tc^2=\sum_{s\in\Vc}\abs{u}_{\Tc,s}^2
\]

\section{The MPSA Finite Volume discretization}\label{sec3}

In this section, we will utilize the spaces defined in Section~\ref{sec2} to
establish a cell-centered finite volume method for elasticity. The method
presented herein is a slight generalization of the MPSA O-method as it
was defined in \cite{nor14b}. The construction is inspired by, but generalizes in
necessary aspects, the hybrid finite volume methods in \cite{age08,age10}.

\subsection{Discrete mixed variational problem}\label{sec3.1}

Since we are dealing with the vector equation~\eqref{eq1.1}, we will seek
solutions $\ub$ in the discrete vector-valued spaces $\Hbc_{\chi}=(\Hc_{\chi})^d$,
where $\chi\in\{\Cc,\Dc,\Tc\}$. These spaces inherit all the definitions of their
scalar counterparts, with the understanding that all inner products are
extended with vector inner products in their definitions. The norms are
in consequence the root of the square of the component-wise norms.

Following \cite{age08}, we will use two notions of discrete gradients. However,
since we will work with both spaces $\Hbc_\Cc$ and $\Hbc_\Dc$, where the latter is
multivalued on internal edges, the precise construction of the gradients
differs from those works. The first gradient is the proper negative
transpose of the divergence operator, and is constructed to yield a
conservative finite volume formulation.

\begin{definition}\label{def3.1}
For each $K\in\Tc$ and each $s\in\Vc_K$ we define the \emph{finite
volume gradient} for all $\ub\in\Hbc_\Cc$:
\begin{equation}\label{eq3.1}
(\widetilde{\nabla}\ub)_K^s=\frac1{m_K^s} \sum_{\sigma\in\Fc_K\cap\Fc_s}m_\sigma^s 
(\aver{\ub}_s^\sigma-\ub_K)\otimes\nb_{K,\sigma}
\end{equation} 
\end{definition}

Here, we denote the vector outer product with $\otimes$, such that each row
of the product matrix contains the approximation to the gradient of the
corresponding component of $\ub$. We construct a second gradient with the
property that it is exact for linear variation in $\ub$ with respect to the
underlying physical space.

\begin{definition}\label{def3.2}
For each $K\in\Tc$ and each $s\in\Vc_K$ we define the \emph{consistent
gradient} for all $\ub\in\Hbc_\Dc$:
\begin{equation}\label{eq3.2}
(\bnabla\ub)_K^s=\sum_{\sigma\in\Fc_K\cap\Fc_s}(\aver{\ub}_{K,s}^\sigma-\ub_K)
\otimes\gb_{K,\sigma}^s
\end{equation} 
\end{definition}

Here we extend the averaging notation in the natural way such that
$\aver{\ub}_{K,s}^\sigma=\frac1{m_s^\sigma}\sum_{\beta\in\Gc_s^\sigma}\omega_\beta u_{K,s}^{\sigma,\beta}$. In
order to satisfy the desired consistency property, we require that $(\bnabla\ub)_K^s$ 
is exact for linear displacements, therefore $\gb_{K,\sigma}^s$ are
defined by the system of equations:
\begin{equation}\label{eq3.3}
\Ib=(\bnabla\xb)_K^s=\sum_{\sigma\in\Fc_K\cap\Fc_s}
(\aver{\xb}_{K,s}^\sigma-\xb_K)\otimes\gb_{K,\sigma}^s
\end{equation} 
Here $\Ib$ is the $d$-dimensional second-order identity tensor.

For all 2D grids and for all 3D grids where no more than three faces
meet at any vertex, equation~\eqref{eq3.3} uniquely determines $\gb_{K,\sigma}^s$ 
\cite{aav02,age10}. This encompasses all grids considered herein.

In the case of both gradients, the symmetric gradient is obtained in
the same way as in the continuous setting by taking the average of the
gradient and its transpose. Furthermore, the equivalent discrete divergence
is the trace of the gradient, e.g.
\[
(\widetilde{\unabla}\ub)_K^s=[(\widetilde{\nabla}\ub)_K^s+{(\smash{\widetilde{\nabla}}\ub)_K^s}^T]/2 
\quad\text{and}\quad
(\widetilde{\nabla}\cdot\ub)_K^s=\tr(\widetilde{\nabla}\ub)_K^s
\]
while also
\[
(\ubnabla\ub)_K^s=[(\bnabla\ub)_K^s+{(\smash{\bnabla}\ub)_K^s}^T]/2 
\quad\text{and}\quad
(\bnabla\cdot\ub)_K^s=\tr(\bnabla\ub)_K^s
\]

We now define our finite volume scheme for linear elasticity, equations~\eqref{eq1.1}, 
through the specification of the following three bilinear forms. The
first bilinear form embodies a discrete form of Hooke's law together
with the finite volume structure of the method, and is analogous to the
weak form stated in equation~\eqref{eq1.2}. Thus we define for $(\ub,\vb)\in\Hbc_\Dc\times\Hbc_\Cc$
\begin{equation}\label{eq3.4}
b_\Dc (\ub,\vb)=\sum_{K\in\Tc}\sum_{s\in\Vc_K}m_K^s \bigl(2\mu_K (\ubnabla 
\ub)_K^s:(\widetilde{\unabla}\vb)_K^s+\lambda_K 
(\bnabla\cdot\ub)_K^s (\widetilde{\nabla}\cdot\vb)_K^s\bigr) 
\end{equation}
The discrete coefficients are given as subcell averages of their
continuous counterparts, e.g.\ $\mu_K=m_K^{-1} \int_K\mu \,d\xb$. The
second bilinear form controls jumps across subcell faces, and is defined
for $(\ub,\wb)\in\Hbc_\Dc\times\Hbc_\Dc$
\begin{equation}\label{eq3.5}
a_\Dc (\ub,\wb)=\sum_{s\in\Vc}\sum_{\sigma\in\Fc_s}\frac{\alpha_s^\sigma}{m_s^\sigma}
\sum_{\beta\in\Gc_\sigma^s}\omega_\beta \dblbr{\ub}_s^{\sigma,\beta}\cdot\dblbr{\wb}_s^{\sigma,\beta} 
\end{equation}
The family of weights $\alpha_s^\sigma$ are assumed to be uniformly bounded,
$0<\alpha^-\le\alpha_s^\sigma\le\alpha^+<\infty$. We retain the freedom to specify $\alpha_s^\sigma$ to
improve the stability of the scheme. Numerically experiments indicate
that the weights $\alpha_s^\sigma$ should be related to the harmonic mean of the
material coefficients $\mu_K$ and $\lambda_K$ \cite{nor14b}.

Finally, we introduce a bilinear form to constrain the discrete unknowns on
subfaces $\ub_{K,s}^{\sigma,\beta}$ to the hyperplane given by $\ub_K$ and $(\bnabla\ub)_K^s$;
thus for all $(\ub,\wb)\in\Hbc_\Dc\times\Hbc_\Dc$ we define
\begin{multline}
c_\Dc(\ub,\wb)=\sum_{K\in\Tc}\sum_{s\in\Vc_K}\sum_{\sigma\in\Fc_s}\sum_{\beta\in\Gc_\sigma^s}
\bigl(\ub_{K,s}^{\sigma,\beta}-\ub_K-(\bnabla\ub)_K^s\cdot(\xb_\beta-\xb_K)\bigr)\\
\cdot
\bigl(\wb_{K,s}^{\sigma,\beta}-\wb_K-(\bnabla\wb)_K^s\cdot(\xb_\beta-\xb_K)\bigr) 
\label{eq3.6}
\end{multline}

The above bilinear forms allow us to define the discrete mixed variational
problem:
Find $(\ub_\Dc,\yb_\Cc,\yb_\Dc )\in\Hbc_\Dc\times\Hbc_\Cc\times\Hbc_\Dc$ such that
\begin{alignat}{2}
&b_\Dc (\ub_\Dc,\vb)=\int_\Omega \fb\cdot\Pc_{\Cc,\Tc} \vb\, d\xb&&\quad\text{for all }\vb\in\Hbc_\Cc\label{eq3.7}\\
&c_\Dc (\ub_\Dc,\wb)=0 &&\quad\text{for all }\wb\in\Hbc_\Dc
\label{eq3.8}
\end{alignat}
and
\begin{equation}
a_\Dc (\ub_\Dc,\wb)+b_\Dc (\wb,\yb_\Cc )+c_\Dc (\wb,\yb_\Dc )=0\quad \text{for all }\wb\in\Hbc_\Dc
\label{eq3.9}
\end{equation}

The solution $\ub_\Dc\in\Hbc_\Dc$ contains the solution satisfying the equations
of elasticity on finite volume form \eqref{eq3.7}, constrained to be piece-wise
linear on sub-cells \eqref{eq3.8}, while the remaining degrees of freedom are
selected to minimize jumps in the solution \eqref{eq3.9}. The component $\yb_\Cc$ and
$\yb_\Dc$ are Lagrange multipliers for the constrained minimization problem,
and will not be of further interest.

We note that equation~\eqref{eq3.7} can be seen as a direct finite volume
formulation of equations~\eqref{eq1.1}, wherein these equations hold in an integral
sense for each cell $K\in\Hc_\Tc$. Conversely, equation~\eqref{eq3.7} can be identified
as a Petrov--Galerkin discretization of equations~\eqref{eq1.2}, wherin the test
functions are chosen as piece-wise constants on the cells $K$. In this
interpretation, the shape functions are defined implicitly by equations
\eqref{eq3.8} and \eqref{eq3.9}. We return to the consistency of equations~\eqref{eq3.7}--\eqref{eq3.9}
in Section~\ref{sec5}.

\subsection{Finite volume formulation}\label{sec3.2}
In this section we identify that the discrete mixed variational problem
\eqref{eq3.7}--\eqref{eq3.9} is equivalent to a finite volume scheme. The forces acting on a
subface $\Tb_{K,s}^\sigma$ are naturally defined from the bilinear form $b_\Dc$ and
Hooke's law (equation~\eqref{eq1.1}b), thus we define for all for all $\ub\in\Hbc_\Dc$
the tractions
\begin{equation}
\Tb_{K,s}^\sigma (\ub)=m_\sigma^s [2\mu_K (\ubnabla\ub)_K^s
+\lambda_K (\bnabla\cdot\ub)_K^s\Ib]\cdot\nb_{K,\sigma} \label{eq3.10}
\end{equation}
We verify by comparison to equation~\eqref{eq3.4} that for all $(\ub,\vb)\in\Hbc_\Dc\times\Hbc_\Cc$
\begin{equation}
b_\Dc(\ub,\vb)=\sum_{K\in\Tc}\sum_{s\in\Vc_K}\sum_{\sigma\in\Fc_K\cap\Fc_s}
\Tb_{K,s}^\sigma(\ub)\cdot(\vb_K-\vb_s^\sigma)\label{eq3.11}
\end{equation}
By considering $\vb$ from the canonical basis of $\Hbc_\Cc$, we can now identify
that the discrete variational mixed formulation \eqref{eq3.7}--\eqref{eq3.9} as equivalent
to the hybrid finite volume method: Find $\ub_\Dc\in\Hbc_\Dc$ such that
\begin{alignat}{2}
&\sum_{\sigma\in\Fc_K}\Tb_K^\sigma (\ub_\Dc)=\int_K\fb\,d\xb&&\quad
\text{for all }K\in\Tc; \label{eq3.12}\\
&\Tb_K^\sigma (\ub_\Dc )=\sum_{s\in\Vc_\sigma}\Tb_{K,s}^\sigma (\ub_\Dc)&&\quad
\text{for all }(K,\sigma)\in\Tc\times\Fc_K; \label{eq3.13}\\
&\Tb_{R,s}^\sigma (\ub_\Dc)=-\Tb_{L,s}^\sigma (\ub_\Dc)&&\quad 
\text{for all }(\sigma,s)\in\Fc_{\mathit{int}}\times\Vc_\sigma
\notag\\[-.7ex]
&&&\quad\text{with }\{R,L\}=\Tc_\sigma; \label{eq3.14}\\
&\ub_{K,s}^{\sigma,\beta}-\ub_K-(\bnabla\ub)_K^s\cdot(\xb_\beta-\xb_K)=\Ob
&&\quad
\text{for all }
(K,s,\sigma)\in\Tc\times\Vc_K\times(\Fc_s\cap\Fc_K)
\notag\\[-.7ex]
&&&\quad\text{and }\beta\in\Gc_\sigma^s;
\label{eq3.15}\\
&\ub_\Dc=\argmin_{\ub\in\Ub_1}a_\Dc (\ub,\ub)&&\quad 
\text{where $\Ub_1\subset\Hbc_\Dc$ is the
set of}
\notag\\[-.7ex]
&&&\quad\text{functions satisfying \eqref{eq3.12}--\eqref{eq3.15};} 
\label{eq3.16}
\end{alignat}

To be precise, equation~\eqref{eq3.12} follows from equation~\eqref{eq3.7} by choosing test
functions $\vb\in\Hbc_\Tc$. Equation~\eqref{eq3.13} follows from the definition of the space
$\Hbc_\Cc$, which contains a single degree of freedom on each edge, and that the
right-hand side of equation~\eqref{eq3.7} is zero for test functions associated
with faces of the grid. Similarly, equation~\eqref{eq3.14} follows from the
continuity of the face variables in $\Hbc_\Cc$ and the same property of 
equation~\eqref{eq3.7}. Finally, equations~\eqref{eq3.15} and~\eqref{eq3.16} are direct counterparts of
the constraint equation~\eqref{eq3.8} and the interpretation of equation~\eqref{eq3.9}
as the Euler-Lagrange equation for a constrained minimization problem.

We will see in Section~\ref{sec4} that due to the particular structure of the
discrete differential operators, and the local definition of the forces
in equation~\eqref{eq3.10}, the minimization problem in equation~\eqref{eq3.16} has a
unique solution in terms of the variables $\ub_\Tc\in\Hbc_\Tc\subset\Hbc_\Dc$, the set of
variables associated with the cell centers.

Assume for the moment (we will return to this point in the next section)
that for each $s\in\Vc$ we can define the local projection operator
$\Pi_{FV,s}\colon\Hbc_{\Tc,s}\to\Hbc_{\Dc,s}$, as the solution of
\begin{equation}
\Pi_{FV,s} \ub_{\Tc,s}=\argmin_{\ub\in\Ub_s}a_\Dc (\ub,\ub)
\label{eq3.17}
\end{equation}
Where $\Ub_s\subset\Hbc_{\Dc,s}$ are the spaces that satisfy the constraints of
\eqref{eq3.13}--\eqref{eq3.15} with $\ub_{\Tc,s}$ given. It follows that we can construct explicit,
local expressions for the forces $\Tb_{K,s}^\sigma$ using the expression given
in equation~\eqref{eq3.10}:
\begin{equation}
\Tb_{K,s}^\sigma (\ub_{\Tc,s})=\Tb_{K,s}^\sigma (\Pi_{FV,s} \ub_{\Tc,s})=
\sum_{K'\in\Tc_s}t_{K,K',s,\sigma} \ub_{K'} \label{eq3.18}
\end{equation}
The local coefficient tensors $t_{K,K',s,\sigma}$ are referred to as subface
stress weight tensors, and generalize the notion of transmissibilities
from the scalar diffusion equation \cite{nor14b}. We infer from equation~\eqref{eq3.14}
that $t_{K,K',s,\sigma}=-t_{K',K,s,\sigma}$. Furthermore, we have from equation~\eqref{eq3.13}, 
that also the face stress weight tensors can be calculated, with
\begin{equation}
\Tb_K^\sigma (\ub_\Tc)=\sum_{s\in\Vc_\sigma}\sum_{K'\in\Tc_s}t_{K,K',s,\sigma} \ub_{K'} 
\label{eq3.19}
\end{equation}
Combining equations~\eqref{eq3.12} and~\eqref{eq3.19}, we arrive at the cell-centered
finite volume scheme expressed in terms of cell-centered variables
only. This scheme is identical to the scheme presented as \emph{MPSA O-method
(general)} in \cite{nor14b}.

\section{Local problems, Korn's inequality, and coercivity}\label{sec4}
Our goal is to show that the discrete mixed variational problem
\eqref{eq3.7}--\eqref{eq3.9} is well-posed. This requires four steps. First, we formalize the
discussion in section~\ref{sec3.2} using a variational multiscale framework to state
variational problem only in terms of variables in the space $\Hbc_\Tc$, exploiting
local operators which are defined through local problems. Secondly,
we show the stability of these local problems. Thereafter, we arrive
at a discrete Korn's inequality through a projection onto piece-wise
linear space on each subcell. Finally, we establish that coercivity,
and thus wellposedness, of the full problem can be verified based on
local coercivity criteria.

\subsection{A non-mixed discrete variational formulation}\label{sec4.1}
We use an approach similar to the variational multiscale methods \cite{hug98}
as applied to mixed problems \cite{arb04,nor09}, in that we split the mixed problem
\eqref{eq3.7}--\eqref{eq3.9} into two coupled problems. We introduce $\Hbc_\Dc^\Fc\subset\Hbc_\Dc$ denoting the
variables associated with cell faces, such that $\Hbc_\Dc=\Hbc_\Tc\times\Hbc_\Dc^\Fc$. Identifying
now $\Hbc_\Tc$ as the space of coarse variables, and $\Hbc_\Dc^\Fc$ as the space of
fine variables, we thus consider the problem: Find 
$(\ub_\Tc,\ub_F,\yb_\Cc,\yb_\Dc)\in\Hbc_\Tc\times\Hbc_\Dc^\Fc\times\Hbc_\Cc\times\Hbc_\Dc$ 
such that (\emph{coupled coarse problem}):
\begin{equation}
b_\Dc (\{\ub_\Tc,\ub_\Fc\},\vb)=\int_\Omega\fb\cdot\Pi_\Tc \{\vb,\Ob_\Fc\}\,d\xb 
\quad\text{for all }
\vb\in\Hbc_\Tc 
\label{eq4.1}
\end{equation}
And (\emph{coupled mixed fine problem})
\begin{alignat}{2}
&b_\Dc (\{\Ob_\Tc,\ub_\Fc\},\vb)=-b_\Dc (\{\ub_\Tc,\Ob_\Fc\},\vb)&&\quad 
\text{for all }\vb\in\Hbc_\Cc^\Fc\label{eq4.2}\\
&c_\Dc (\{\Ob_\Tc,\ub_\Fc\},\wb)=-c_\Dc (\{\ub_\Tc,\Ob_\Fc\},\wb)&&\quad 
\text{for all }\wb\in\Hbc_\Dc
\label{eq4.3}\\
&
\begin{aligned}[b]
a_\Dc(\{\Ob_\Tc,\ub_\Fc\},\wb) &+ b_\Dc (\wb,\yb_\Cc)\\ &+ c_\Dc(\wb,\yb_\Dc)=
-a_\Dc(\{\ub_\Tc,\Ob_\Fc\},\wb)
\end{aligned}
&&\quad 
\text{for all }\wb\in\Hbc_\Dc \label{eq4.4}
\end{alignat}
Note that there is no integral term on the right hand sides of equations
\eqref{eq4.2}--\eqref{eq4.4} since $\Pi_\Tc \{\Ob_\Pc,\vb\}=\Ob$. Furthermore only the fine-scale problem
is on mixed form.

As observed in section~\ref{sec3}, the aim is to resolve the mixed fine problem
locally, which in the present context is realized by interchanging sums in
the definition of the operators \eqref{eq3.4}--\eqref{eq3.6} to observe that for $\chi\in\{a,c\}$
and for $(\ub,\vb,\wb)\in\Hbc_\Dc\times\Hbc_\Cc\times\Hbc_\Dc$
\begin{equation}
\chi_\Dc (\ub,\wb)=\sum_{s\in\Vc}\chi_{\Dc,s} (\ub,\wb) 
\quad\text{and}\quad 
b_\Dc(\ub,\vb)=\sum_{s\in\Vc}b_{\Dc,s} (\ub,\vb) \label{eq4.5}
\end{equation}
where the local bilinear forms are defined as
\begin{align}
b_{\Dc,s} (\ub,\vb)&=\sum_{K\in\Tc_s}m_K^s \bigl(2\mu_K^s (\ubnabla\ub)_K^s:
(\widetilde{\unabla}\vb)_K^s+\lambda_K^s (\bnabla\cdot\ub)_K^s (\widetilde{\nabla}\cdot\vb)_K^s\bigr)
\label{eq4.6}\\
\begin{split}
c_{\Dc,s}(\ub,\wb)
&=\sum_{K\in\Tc_s}\sum_{\sigma\in\Fc_s}\sum_{\beta\in\Gc_\sigma^s}
\bigl(\ub_{K,s}^{\sigma,\beta}-\ub_K-(\bnabla\ub)_K^s\cdot(\xb_\beta-\xb_K)\bigr)
\\
&\hskip3cm
\cdot
\bigl(\wb_{K,s}^{\sigma,\beta}-\wb_K-(\bnabla\wb)_K^s\cdot(\xb_\beta-\xb_K)\bigr) 
\end{split}
\label{eq4.7}
\end{align}
and
\begin{equation}
a_{\Dc,s} (\ub,\wb)=\sum_{\sigma\in\Fc_s}\frac{\alpha_s^\sigma}{m_s^\sigma}
\sum_{\beta\in\Gc_\sigma^s}\omega_\beta \dblbr{\ub}_s^{\sigma,\beta}\cdot\dblbr{\wb}_s^{\sigma,\beta} 
\label{eq4.8}
\end{equation}
This defines the local solution operators $\Pi_{FV,s}$, which were introduced
in section~\ref{sec3.2}, by (\emph{local mixed problem}): For all $\ub_\Tc\in\Hbc_{\Tc,s}$, find
$(\Pi_{FV,s} \ub_\Tc,\yb_\Cc,\yb_\Dc)\in\Hbc_{\Dc,s}^\Fc\times\Hbc_{\Cc,s}\times\Hbc_{\Dc,s}$ which satisfies
\begin{alignat}{2}
&b_{\Dc,s}(\{\Ob_\Tc,\Pi_{FV,s} \ub_\Tc\},\vb)=b_{\Dc,s} (\{\ub_\Tc,\Ob_\Fc\},\vb)&&\quad 
\text{for all }\vb\in\Hbc_{\Cc,s}^\Fc \label{eq4.9}\\
&c_{\Dc,s}(\{\Ob_\Tc,\ub_\Fc\},\wb)=-c_{\Dc,s}(\{\ub_\Tc,\Ob_\Fc\},\wb)&&\quad 
\text{for all }\wb\in\Hbc_{\Dc,s} \label{eq4.10}\\
&\begin{aligned}[b]
&a_{\Dc,s} (\{\Ob_\Tc,\Pi_{FV,s} \ub_\Tc\},\wb) \\
&\qquad + b_{\Dc,s} (\wb,\yb_\Cc) + c_{\Dc,s} (\wb,\yb_\Dc)=a_{\Dc,s}(\{\ub_\Tc,\Ob_\Fc\},\wb)
\end{aligned}
&&\quad 
\text{for all }\wb\in\Hbc_{\Dc,s} \label{eq4.11}
\end{alignat}
Existence and uniqueness of solutions to \eqref{eq4.9}--\eqref{eq4.11} is treated in 
Lemma~\ref{lem4.1} below. We note that the local projections are linear, and construct
the global finite volume projection $\Pi_{FV}\colon\Hbc_\Tc\to\Hbc_\Dc^\Fc$ by
\begin{equation}
\Pi_{FV} \ub_\Tc=\sum_{s\in\Vc}\Pi_{FV,s} \ub_\Tc 
\label{eq4.12}
\end{equation}
Using the finite-volume projection, we establish the \emph{decoupled coarse
problem}: Given $\Pi_{FV}$, find $\ub_\Tc\in\Hbc_\Tc$ such that:
\begin{equation}
b_\Dc \bigl(\Pi_{FV} \ub_\Tc,\Pi_\Cc(\Pi_{FV}\vb_\Tc)\bigr)=\int_\Omega\fb\cdot\Pi_\Tc\{\vb,\Ob_\Fc\}\,d\xb\quad
\text{for all }\vb\in\Hbc_\Tc \label{eq4.13}
\end{equation}
Equations \eqref{eq4.9}--\eqref{eq4.13} are algebraically equivalent to the original
formulation given in equations \eqref{eq3.7}--\eqref{eq3.8}. However, the present formulation
has the advantage that wellposedness can be established in steps. In
particular, the mixed problems \eqref{eq4.9}--\eqref{eq4.11} are all local, so that we avoid
considering a global inf-sup type condition. Furthermore, the global
problem is similar in structure to a standard Galerkin approximation to
the original system \eqref{eq1.1}, and for this problem we will establish (local)
conditions to ensure coercivity.

\subsection{Well-posedness of local mixed problems}\label{sec4.2}
The local mixed problems are given by equations \eqref{eq4.9}--\eqref{eq4.11}, and define
the finite volume projection.

\begin{lemma}\label{lem4.1}
The solutions of the local mixed problems \eqref{eq4.9}--\eqref{eq4.11} exist and
are unique.
\end{lemma}

\begin{proof}
Since $a_{\Dc,s}$ is a symmetric, positive semi-definite bilinear
form, the system is equivalent to a constrained minimization problem,
to which existence of solutions is guaranteed. In particular, for $\ub_\Tc=\Ob$,
it is clear that $\Pi_{FV,s} \ub_\Tc=\Ob$ satisfies the constraints \eqref{eq4.9}--\eqref{eq4.10}
and has the minimum energy $a_{\Dc,s} (\Pi_{FV,s} \Ob_\Tc,\Pi_{FV,s}\Ob_\Tc)=0$. Thus
for $\ub_\Tc=\Ob$, the space of minimizers is isomorphic to the space of continuous
piecewise linear functions on $\Tc_s$. However, all cells $K\in\Tc$ are star-shaped
with respect to $\xb_K$, thus $(\aver{\xb}_{K,s}^\sigma-\xb_K)\cdot\nb_{\sigma,s}>0$, from
which it follows that the constraints \eqref{eq4.9} are linearly independent
to the null-space of $a_{\Dc,s} (\{\Ob_\Tc,\ub\},\{\Ob_\Tc,\ub\})$. Uniqueness follows
since it is straight-forward to verify that the null-space of $a_{\Dc,s}
(\{\Ob_\Tc,\ub\},\{\Ob_\Tc,\ub\})$ has at most $d^2$ degrees of freedom, while at least $d^2$
constraints \eqref{eq4.9} are independent. 
\end{proof}

Since the local mixed problem \eqref{eq4.9}--\eqref{eq4.11} is finite-dimensional, the norm
of the projection operator is finite and we define:

\begin{definition}\label{def4.2}
For every $s\in\Vc$, the local mixed problem \eqref{eq4.9}--\eqref{eq4.11} has
a unique solution by lemma~\ref{lem4.1}, and we define the norm of the of the
solution operator $\Pi_{FV,s}$ as $\theta_{1,s}$, such that
\begin{equation}
\abs{\Pi_{FV,s} \ub}_{\Dc,s}\le\theta_{1,s} \abs{\ub}_{\Tc,s}\quad 
\text{for all }\ub\in\Hbc_\Tc
\label{eq4.14}
\end{equation}
\end{definition}

The coefficient $\theta_{1,s}$ will in general be dependent on the geometry of
the mesh $\Dc$ and parameter functions $\mu$ and $\lambda$, but independent of the mesh
size $h$ due to the scaling invariance of the norm. We define the maximum
over the local norms as $\theta_1=\max_{s\in\Vc}\theta_{1,s}$.

\subsection{Korn's inequality}\label{sec4.3}

We will need a discrete Korn's inequality to show coercivity of the
method. Since the constraints~\eqref{eq4.10} guarantee that the projections
$\Pi_{FV,s}$ are consistent with a piecewise linear approximation on the
subcells, we are inspired to consider Korn's inequality in this setting:

\begin{definition}\label{def4.3}
Let the projection operator $\Pi_{L^2}\colon\Hbc_\Dc\to(L^2(\Omega))^d$
be defined such that for all $(K,s)\in\Tc\times\Vc$ and $\xb\in K$, we have 
$\Pi_{L^2}\ub(\xb)=\ub_K+(\bnabla\ub)_K^s\cdot(\xb-\nobreak\xb_K)$.
\end{definition}

We note that this projection is particularly appropriate for functions
in $\ub\in\Pi_{FV,s} \Hbc_\Tc$, since equations~\eqref{eq4.7} and \eqref{eq4.10} assure that
in this space the face variables satisfy $\ub_{K,s}^{\sigma,\beta}=\Pi_{L^2}
\ub(\xb_{K,s}^{\sigma,\beta})$ for all cells $K\in\Tc$, faces $\sigma\in\Fc_K$, corners
$s\in\Vc_K\cap\Vc_\sigma$ and quadrature points $\beta\in\Gc_\sigma^s$.

For the space of piecewise linear functions we may thus recall the
appropriate Korn's inequality \cite{bre03}

\begin{lemma}\label{lem4.4}
For functions $\ub\in\Pi_{FV,s} \Hbc_\Tc$ it holds that
\begin{multline*}
\sum_{K\in\Tc}\sum_{s\in\Vc_K}m_K^s((\bnabla\ub)_K^s)^2\\ 
\le c_K
\biggl[\norm{\Pi_{L^2 }\ub}_{L^2}^2+\sum_{K\in\Tc}\sum_{s\in\Vc_K}m_K^s 
((\ubnabla\ub)_K^s)^2+\sum_{\sigma\in\Fc_s\bigcap\Fc_K}\frac{m_K^s}{d_{K,\sigma}^2}
\frac1{m_s^\sigma}\sum_{\beta\in\Gc_s^\sigma}\omega_\beta(\dblbr{\ub}_s^{\sigma,\beta})^2\biggr]
\end{multline*}
\end{lemma}

\begin{proof}
The left-hand side is identical to the broken $H_1$ semi norm of
$\Pi_{L^2}\ub$, while the jump term bounds also the surface $L^2$ norm of jumps
internal to cells due to mesh regularity and continuity at the points $\xb_K$
for all cells $K\in\Tc$. Lemma~\ref{lem4.4} therefore follows from equation (1.18) in
\cite{bre03}.
\end{proof}

\subsection{Local coercivity conditions}\label{sec4.4}

We will derive local coercivity conditions which guarantee the coercivity
of the bilinear form on the reduced subspace as given by equation~\eqref{eq4.13}. 
The reduced subspace is critical, as $b_{\Dc,s}$ is \emph{not} coercive on
the full space $\Hbc_\Dc$ (even without the introduction of a full norm). As a
consequence, coercivity will depend on the local finite volume projections
$\Pi_{FV,s}$, and we state the following assumption on the solution of the
local mixed problems.

\begin{assumption}\label{ass4.5}
For every vertex $s\in\Vc$, there exists a constant
$\theta_{2,s}\ge\theta_2>0$ such that the bilinear form $b_{\Dc,s}$ and the projection
$\Pi_{FV,s}$ satisfy for all $\ub\in\Pi_{FV,s}\Hbc_\Tc/\Rfr(\Omega)$
{\emergencystretch10pt
\begin{equation}
b_{\Dc,s} (\ub,\Pi_\Cc\ub)\ge\theta_{2,s}
\biggl(\abs{\ub}_{b_{\Dc,s}}^2+\sum_{K\in\Tc_s}\sum_{\sigma\in\Fc_s\bigcap\Fc_K}\frac{m_K^s}{d_{K,\sigma}^2}\frac1{m_s^\sigma}
\sum_{\beta\in\Gc_s^\sigma}\omega_\beta (\dblbr{\ub}_s^{\sigma,\beta})^2\biggr)
\label{eq4.15}
\end{equation}}
Where the local energy semi-norm is associated with the symmetrized
bilinear form
\[
\abs{\ub}_{b_{\Dc,s}}^2=\sum_{K\in\Tc_s}m_K^s \bigl(2\mu_K((\ubnabla\ub)_K^s)^2
+\lambda_K((\bnabla\cdot\ub)_K^s)^2\bigr) 
\]
\end{assumption}

This assumption can be verified locally while assembling the
discretization, and moreover it can be verified \emph{a priori} for certain
classes of meshes, see Section~\ref{sec7}. We recall two useful stability estimates.

\begin{lemma}\label{lem4.6}
\textup{a)} There exists a constant $c_\Dc$, only dependent on the regularity
of $\Dc$, such that for all $\ub\in\Hbc_\Dc$
\[
\abs{\ub}_\Dc^2\le c_\Dc \biggl(\sum_{K\in\Tc}\sum_{s\in\Vc_K}m_K^s ((\bnabla\ub)_K^s
)^2 +\sum_{\sigma\in\Fc_s\bigcap\Fc_K}\frac{m_K^s}{d_{K,\sigma}^2}\frac1{m_s^\sigma}
\sum_{\beta\in\Gc_s^\sigma}\omega_\beta (\dblbr{\ub}_s^{\sigma,\beta} )^2 \biggr)
\]
\textup{b)} Furthermore, there exists a constant $c_{L^2}$, only dependent on the
regularity of $\Dc$, such that for all $\ub\in\Pi_{FV,s}\Hbc_\Tc$
\[
\norm{\Pi_{L^2}\ub}_{L^2/\Rfr(\Omega)}^2\le c_{L^2}
\biggl(\sum_{K\in\Tc}\sum_{s\in\Vc_K}m_K^s ((\ubnabla\ub)_K^s)^2
+\sum_{\sigma\in\Fc_s\bigcap\Fc_K}\frac{m_K^s}{d_{K,\sigma}^2}\frac1{m_s^\sigma}
\sum_{\beta\in\Gc_s^\sigma}\omega_\beta(\dblbr{\ub}_s^{\sigma,\beta})^2\biggr)
\]
\end{lemma}

\begin{proof}
The first inequality follows readily from the definition of
the semi-norm and the consistent gradient $(\bnabla\ub)_K^s$. The second
inequality is shown by contradiction. Let the right-hand side be zero. Then
$(\ubnabla\ub)_K^s=\Ob$, and due to \eqref{eq3.8} the function $\ub$ thus lies in
the space of rigid body motions on each subcell $(K,s)$. However, since
the second-order Gauss quadrature is exact for linear functions, $\ub$ is in
the space $\Rfr(\Omega)$. But then also the left-hand side of the inequality is
zero. The scaling of the norms then guarantees that the constant $c_{L^2}$
is independent of $h$.
\end{proof}

We are now prepared to state the global coercivity result for the method

\begin{theorem}\label{the4.7}
For given parameter fields $\mu$, $\lambda$, and mesh $\Dc$, let assumption~\ref{ass4.5} 
hold. Then the coarse variational problem~\eqref{eq4.13} is coercive in the
sense that it satisfies if \textup{a)} $\Gamma_D$ is measurable then for all $\ub\in\Hbc_\Tc$
and \textup{b)} if $\Gamma_N=\partial\Omega$ then for all $\ub\in\Hbc_\Tc/\Rfr(\Omega)$
\[
b_\Dc \bigl(\Pi_{FV} \ub,\Pi_\Cc (\Pi_{FV} \ub)\bigr)\ge\Theta\abs{u}_\Tc^2
\]
Where the constant $\Theta$ is dependent on the mesh triplet $\Dc$ but does not
scale with $h$, and is bounded below by:
\[
\Theta\ge \frac{\theta_2}{dc_\Dc c_K (1+c_{L^2})}\min\left(2\mu,\frac1{1+(c_K(1+c_{L^2}))^{-1}}\right)
\]
\end{theorem}

\begin{proof}
By the definition of the local bilinear forms, Assumption~\ref{ass4.5} and
the lower bounds on the coefficient functions $\mu$ and $\lambda$ we have for all
$\ub\in\Pi_{FV,s} \Hbc_\Tc$
\[
b_\Dc(\ub,\ub)\ge\theta_2 \biggl(2\underline{\mu}\sum_{K\in\Tc}\sum_{s\in\Vc_K}m_K^s 
((\ubnabla\ub)_K^s)^2
+\sum_{\sigma\in\Fc_s\bigcap\Fc_K}\frac{m_K^s}{d_{K,\sigma}^2}\frac1{m_s^\sigma}
\sum_{\beta\in\Gc_s^\sigma}\omega_\beta (\dblbr{\ub}_s^{\sigma,\beta})^2\biggr)
\]

We denote $c_\mu=\min(\underline{2\mu},1-\xi)$, for any $0<\xi<1$, and now due to Lemma~\ref{lem4.4} 
and \ref{lem4.6}a) we obtain
\[
b_\Dc (\ub,\ub)\ge\frac{\theta_2 c_\mu}{c_K(1+c_{L^2})}
\sum_{K\in\Tc}\sum_{s\in\Vc_K}m_K^s ((\bnabla\ub)_K^s)^2
+\theta_2\xi\sum_{\sigma\in\Fc_s\bigcap\Fc_K}\frac{m_K^s}{d_{K,\sigma}^2}\frac1{m_s^\sigma}
\sum_{\beta\in\Gc_s^\sigma}\omega_\beta (\dblbr{\ub}_s^{\sigma,\beta})^2
\] 
Utilizing Lemma~\ref{lem4.6}b) together with the basic norm inequalities stated
in Section~\ref{sec2.1} completes the proof when $\xi$ is chosen to maximize $\Theta$. 
\end{proof}

\section{Convergence of the method}\label{sec5}

Convergence of the scheme as given in Section~\ref{sec3} is now established for all
grid sequences satisfying a uniform coercivity bound. In particular, weak
convergence to some function $\widetilde{\ub}\in(H^1)^d$ of the discrete solution $\Pi_\Tc
\ub_\Dc$ and its gradient $\nabla_\Dc \ub_\Dc$ (defined below), follows immediately from
previous work \cite{age10}. The strong convergence of the gradient $\bnabla_\Dc \ub_\Dc$,
and establishing that $\widetilde{\ub}$ is a weak solution of Equations~\eqref{eq1.1}, requires
invoking the finite-volume projection $\Pi_{FV}$ due to the discontinuous nature
of the discrete space $\Hbc_\Dc$ and the mixed formulation of the saddle-point
problems used in the current work. We structure this section accordingly.

\begin{definitions}\label{defs5.1} 
We consider the following continuous extensions of the
cell-average finite volume gradients for discrete functions $\ub\in\Hbc_\Dc$:
\begin{equation}
\nabla_\Dc \ub(\xb)=(\widetilde{\nabla}\ub)_K\quad\text{for $K\in\Tc$, where $\xb\in K$}. \label{eq5.1}
\end{equation}
Furthermore, we consider the continuous extension of the consistent
gradient
\begin{equation}
\bnabla_\Dc \ub(\xb)=(\bnabla\ub)_K^s\quad\text{for $K\in\Tc$, $s\in\Tc$ where $\xb\in K_s$}.
\label{eq5.2}
\end{equation}
\end{definitions}

Note that from its definition, the discrete gradient satisfies the
stability estimate
\[
\norm{\bnabla_\Dc\ub}_{L_2}\le g_0 \norm{\ub}_\Dc
\]
Where $g_0=\max_{K\in\Tc,\sigma\in\Fc_K,s\in\Vc_\sigma}\abs{\gb_{K,\sigma}^s}/d_{K,\sigma}$, 
which depends on the regularity of the grid but not on $h$.

We now recall the following result:

\begin{lemma}\label{lem5.2}
Let $\Dc_n$ be a family of regular discretization triplets (in
the sense that mesh parameters remain bounded) such that $h_n\to0$, as
$n\to\infty$. Furthermore, let $\theta_1$ and $\Theta$ be bounded independently of $n$. Then
for all $n$, the solution $\ub_n$ of equations \eqref{eq3.7}--\eqref{eq3.8} exist and are unique,
there exists $\widetilde{\ub}\in(H_1 (\Omega))^d$, and up to a subsequence (still denoted
by $n$) $\Pi_\Tc \ub_n\to\widetilde{\ub}$ converges in $(L^q(\Omega))^d$, for $q\in[1,2d/(d-2+\epsilon))$
as $h_n\to0$. Finally, the cell-average finite volume gradient $\bnabla_\Dc
\ub_n$ converges weakly to $\nabla\widetilde{\ub}$ in $(L^2(\Omega))^{d^2}$.
\end{lemma}

\begin{proof}
The proof follows immediately from the coercivity of the scheme
(Section~\ref{sec4}) and the compactness arguments detailed in \cite{eym07} and \cite{age10}. 
\end{proof}

We also state the strong convergence of the consistent gradient $\bnabla_\Dc
\ub(x)$. This result was achieved for the scalar diffusion equation in \cite{age10},
where the discrete solution lies in $\Hc_\Cc$. We can extend their calculation
to the present case as summarized below.

\begin{lemma}\label{lem5.3}
Consider the same case as in Lemma~\ref{lem5.2}. Then the consistent
gradient $\bnabla_\Dc \ub_n$ converges strongly to $\nabla\widetilde{\ub}$ in $(L^2 (\Omega))^{d^2}$.
\end{lemma}

\begin{proof}
We need to show that
\begin{equation}
\lim_{n\to\infty}\int_\Omega(\bnabla_\Dc \ub_n-\nabla\widetilde{\ub})^2\,d\xb =0
\label{eq5.3}
\end{equation}
Introducing a function $\phib\in(C^\infty (\Omega))^d$ which approximates $\widetilde{\ub}$,
and using the identity
\[
\bnabla_\Dc \ub_n-\nabla\widetilde{\ub}=\bnabla_\Dc (\ub_n-\Pi_{FV} \Pi_\Tc \phib)
+\bnabla_\Dc (\Pi_{FV} \Pi_\Tc
\phib-\Pi_\Cc \phib)+\bnabla_\Dc (\Pi_\Cc \phib)-\nabla\widetilde{\ub}
\]
we can bound the integral in \eqref{eq5.3} by
\begin{multline}
\int_\Omega(\bnabla_\Dc \ub_n-\nabla\widetilde{\ub})^2\,d\xb\\
\begin{aligned}[b]
\le 3\int_\Omega\bigl(\bnabla_\Dc(\ub_n-\Pi_{FV} \Pi_\Tc \phib)\bigr)^2\,d\xb 
&+3\int_\Omega\bigl(\bnabla_\Dc (\Pi_{FV} \Pi_\Tc \phib-\Pi_\Cc\phib)\bigr)^2\,d\xb\\
&+3\int_\Omega\bigl(\bnabla_\Dc (\Pi_\Cc \phib)-\nabla\widetilde{\ub}\bigr)^2\,d\xb 
\end{aligned}
\label{eq5.4}
\end{multline}

The second term on the right-hand side converges since the projection
operators are exact for linear functions, while the last term vanishes for
$n\to\infty$ \cite{age10}, so it suffices to deal with the first right-hand side term
in equation~\eqref{eq5.4}. However, since $\ub_n=\{\Pi_\Tc u_{\Dc,n},\Pi_{FV} \Pi_\Tc u_{\Dc,n}\}$,
the first term is bounded by the bilinear form $b_\Dc$ due to coercivity.
A straight-forward calculation exploiting that $\phib$ approximates $\widetilde{\ub}$
then verifies that all terms on the right-hand side of equation~\eqref{eq5.4}
converge to zero. 
\end{proof}

The preceding definitions and lemmas allow us to state the main convergence
result for the MPSA method.

\begin{theorem}\label{the5.4}
Consider the same case as in Lemma~\ref{lem5.2}. Then the limit $\widetilde{\ub}\in(H^1(\Omega))^d$ of the discrete mixed variational problem \eqref{eq3.7}--\eqref{eq3.9},
and consequently the MPSA O-method, is the unique weak solution of the
problem \eqref{eq1.1}.
\end{theorem}

\begin{proof}
Lemmas~\ref{lem5.2} and \ref{lem5.3} establish that the limit $\widetilde{\ub}\in(H^1(\Omega))^d$
exists, and the appropriate notion of convergence of the discrete
gradients. It remains to show that $\widetilde{\ub}$ is a weak solution of problem
\eqref{eq1.1}. Uniqueness then follows from the uniqueness of weak solutions to
\eqref{eq1.1}. Recall again that $\ub_n=\Pi_{FV}\ub_{\Tc,n}$. Then due to the stability of
the projections and the strong and weak convergence of $\bnabla_\Dc \ub_n$ and
$\nabla_\Dc \ub_n$, respectively, it follows that for all $\ub,\vb\in(C^\infty (\Omega))^d$
\begin{equation}
\lim_{n\to\infty}b_\Dc (\Pi_{FV} \Pi_\Tc \ub,\Pi_\Cc \Pi_{FV} \Pi_\Tc\vb)
=\int_\Omega 2\mu\unabla \ub:\unabla \vb + \lambda(\nabla\cdot\ub)(\nabla\cdot\vb)\,d\xb 
\label{eq5.5}
\end{equation}
and
\begin{equation}
\lim_{n\to\infty}\int_\Omega\fb\cdot\Pi_\Tc\{\Pi_\Tc\vb,\Ob_\Fc\}\,d\xb
=\int_\Omega\fb\cdot\vb\,d\xb 
\label{eq5.6}
\end{equation}
However, since \eqref{eq5.5} and \eqref{eq5.6} are the terms weak form of equation~\eqref{eq1.1},
as stated in equations~\eqref{eq1.2}, and since $C^\infty$ is dense in $H^1$, it follows
that the limit $\widetilde{\ub}$ is the weak solution to \eqref{eq1.1}. 
\end{proof}

\section{Comments on the method}\label{sec6}

In this section we present various comments on the developments of
sections~\ref{sec2}--\ref{sec6}. In particular, we comment on 1) the application to pure
Dirichlet problems with homogeneous coefficients, 2) Reduced quadrature on
simplex grids, 3) the corresponding finite volume for the scalar diffusion
equation, and 4) a related finite difference method.

\subsection{Homogeneous Dirichlet problems}\label{sec6.1}

For the special case where $\Gamma_D=\partial\Omega$, and the Lam\'e coefficients $\mu$
and $\lambda$ are constant on $\Omega$, it is well known that equations~\eqref{eq1.1} can be
re-written by integration-by-parts as \cite{kik88}:
\begin{equation}
\begin{aligned}
&\nabla\cdot\sig=\fb&&\text{in }\Omega\\ 
&\sig=\mu\nabla\ub+(\lambda+\mu)(\nabla\cdot\ub)\Ib&&\text{in }\Omega\\
&\ub=\Ob&&\text{on }\partial\Omega
\end{aligned}
\label{eq6.1}
\end{equation}

This form is locally coercive, since the symmetrized gradient does not
appear. We may proceed to discretize equation~\eqref{eq6.1} as in the preceding
sections, but with the bilinear form defined in equation~\eqref{eq3.4} replaced by:
\begin{equation}
b_\Dc (\ub,\vb)=\sum_{K\in\Tc}\sum_{s\in\Vc_K}m_K^s \bigl(\mu(\bnabla\ub)_K^s:
(\widetilde{\nabla}\ub)_K^s+(\lambda+\mu)(\bnabla\cdot\ub)_K^s(\widetilde{\nabla}\cdot\ub)_K^s\bigr) 
\label{eq6.2}
\end{equation}

The resulting numerical method is also locally coercive and the
Korn's inequality (section~\ref{sec4.3}) is not needed to show coercivity of
the method. While the local coercivity assumption~\ref{ass4.5} is still needed,
this locally coercive formulation will due to the absence of Korn's
inequality have a more favorable global coercivity constant.

\subsection{Reduced integration on simplex grids}\label{sec6.2}

It is possible to consider using lower-order quadrature by choosing a
smaller set of Gauss points $\Gc_\sigma^s$. In particular, choosing a single Gauss
quadrature point leads to the MPSA vector analog of the classical MPFA
O($\eta$) methods for the scalar equation, where the $\eta$ is a parameterization
of the location of the Gauss point \cite{aav02,aav96}. In this setting the quadrature
point will act as a point of strong continuity over the faces $\sigma\in\Fc$
between the (linear) solution in the two adjacent subcells $\Tc_\sigma$. On general
grids, this does \emph{not} lead to a well-posed discretization, since the local
mixed system \eqref{eq4.7}--\eqref{eq4.8} is no longer well-posed: There exists for this
case a non-trivial element of the kernel of $a_{\Dc,s}$ which satisfies the
constraints given by $b_{\Dc,s}$ and $c_{\Dc,s}$.

However, for simplex grids with strictly acute corners, numerical
experiments indicate that the local mixed system remains well-posed
\cite{nor14b}. In this case, it is particularly attractive to consider the MPSA
O($1/3$) method. This parameterization implies that for all $(K,s)\in\Tc$, $\Vc_K$ the
Gauss quadrature points $\Gc_\sigma^s$ are chosen such that the points $\xb_K$, $\xb_\beta$
(for all $\beta=\Gc_\sigma^s$ with $\sigma\in\Fc_K\cap\Fc_s$) and the location of the vertex
$s$ for a parallelogram in 2D and a parallelepiped in 3D. In this case,
a straight-forward calculation shows that the two discrete gradients
coincide, since
\begin{equation}
\gb_{K,\sigma}^s=\frac{m_\sigma^s}{m_K^s}\times\nb_{K,\sigma}
\label{eq6.3}
\end{equation}
Consequently, the bilinear form $b_\Dc$ is symmetric \cite{kla08,eig05,age10}.

\subsection{Corresponding discretization for the scalar diffusion equation}\label{sec6.3}

The discretization analyzed herein can be directly applied to the scalar
diffusion equation, and represents a new method in that context.

For the scalar diffusion equation, the analysis presented herein simplifies
somewhat. In particular, the use of Korn's inequality can be omitted
(as was also the case in Section~\ref{sec6.1}), since the scalar bilinear form
is locally coercive. Thus the local coercivity conditions can be stated
directly, and the global coercivity of the method can be obtained by
simple summation. The convergence of the method then follows by identical
arguments to those used in Section~\ref{sec5}.

\subsection{A related finite difference method}\label{sec6.4}

The local coercivity Assumption~\ref{ass4.5} can be avoided by considering
the related finite difference method. Indeed, consider the symmetric
bilinear form
\begin{equation}
b_\Dc^* (\ub,\vb)=\sum_{K\in\Tc}\sum_{s\in\Vc_K}m_K^s \bigl(2\mu_K (\ubnabla\ub)_K^s:
(\ubnabla\vb)_K^s+\lambda_K (\bnabla\cdot\ub)_K^s (\bnabla\cdot\vb)_K^s\bigr) 
\label{eq6.4}
\end{equation}
Then we can formulate a finite difference method by considering the
discrete mixed variational problem: Find $(\ub_\Dc,\yb_\Cc,\yb_\Dc )\in\Hbc_\Dc\times\Hbc_\Cc\times\Hbc_\Dc$
such that
\begin{equation}
b_\Dc^*(\ub_\Dc,\vb)=\int_\Omega\fb\cdot \Pc_{\Cc,\Tc}\vb\,d\xb\quad\text{for all }\vb\in\Hbc_\Cc
\label{eq6.5}
\end{equation}
together with equations~\eqref{eq3.8} and \eqref{eq3.9} hold. This finite difference method
clearly benefits from all the results given in Sections~\ref{sec4} and \ref{sec5}, however
without the requirement of a local coercivity condition. Furthermore,
the method provides a symmetric discretization, however these benefits
come at the cost of a loss of exact force balance on each grid cell $K\in\Tc$.

\section{Local coercivity assumption for special grids}\label{sec7}

The key assumption in the proof is that there exists a class of grids
satisfying Assumption~\ref{ass4.5}. In this section, we will verify that such
grids exist. To be precise, we give a sufficient condition on the sub-cell
geometry to guarantee coercivity. This will allow us to establish \emph{a priori}
coercivity on regular triangulations of hexagons, squares, and equilateral
parallelograms. Furthermore, we verify unconditional coercivity of the
reduced integration proposed for simplex grids in section~\ref{sec6.2}. Finally,
we discuss the role of the Poisson ratio and locking in Section~\ref{sec7.3}.

\subsection{Vertex-symmetric meshes}\label{sec7.1}

We introduce the following notion, valid for 2D grids.

\begin{definition}\label{def7.1}
We refer to a subcell $(s,K)\in\Vc\times\Tc_s$ as \emph{vertex-symmetric}
if it is symmetric with respect to the line through the points
$(\xb_K,\xb_s)$. Similarly, we refer to a mesh as vertex-symmetric if all
subcells in the mesh triplet are vertex-symmetric.
\end{definition}

Vertex-symmetric meshes include the regular triangulations of hexagons,
squares, and equilateral parallelograms.

We first state a preliminary lemma which resembles the coercivity
assumption:

\begin{lemma}\label{lem7.2}
For every vertex $s\in\Vc$ in a vertex-symmetric mesh, there exists
a constant $\theta_{2,s}^*\ge\theta_2^*>0$ such that the bilinear forms $b_{\Dc,s}$
satisfy for all $\ub\in\Pi_{FV,s} \Hbc_\Tc/\Rfr(\Omega)$
\[
b_{\Dc,s} (\ub,\Pi_\Cc \ub)\ge\theta_{2,s} \abs{\ub}_{b_{\Dc,s}}^2
\]
\end{lemma}

\begin{proof}
Since, due to \eqref{eq4.10}, the space $\Pi_{FV,s} \Hbc_\Tc$ is locally linear
on subcells, and the consistent gradient is by Definition~\ref{def3.2} exact on
linear functions, Lemma~\ref{lem7.2} therefore reduces to showing that for all
matrices $\Pb$, such that $\norm{\Pb}=1$ and for all $(K,s)\in\Tc,\Vc_K$, it holds that
\begin{multline}
\mu_K (\Pb+\Pb^T):(\widetilde{\unabla}(\Pb\xb))_K^s+\lambda_K 
\tr\Pb(\widetilde{\nabla}\cdot(\Pb\xb))_K^s
\\
\ge
\theta_{2,s}\bigl(\mu_K (\Pb+\Pb^T):(\Pb+\Pb^T)+\lambda_K (\tr\Pb)^2\bigr)
\label{eq7.3}
\end{multline}

We consider first the case $\lambda_K=0$, and consider the contradiction
\begin{equation}
(\Pb+\Pb^T):(\widetilde{\unabla}(\Pb\xb))_K^s\le0
\label{eq7.4}
\end{equation}

Using the Definition~\ref{def3.1} we obtain
\begin{equation}
(\Pb+\Pb^T):\sum_{\sigma\in\Fc_K\cap\Fc_s}m_\sigma^s
\bigl[\bigl(\Pb(\aver{\xb}_{K,s}^\sigma-\xb_K)\bigr)\times\nb_{K,\sigma}
+\nb_{K,\sigma}\times\bigl(\Pb(\aver{\xb}_{K,s}^\sigma-\xb_K)\bigr)\bigr]
\le0 
\label{eq7.5}
\end{equation}
Identifying the sum
\begin{equation}
\Sc=\sum_{\sigma\in\Fc_K\cap\Fc_s}m_\sigma^s \bigl[(\aver{\xb}_{K,s}^\sigma-\xb_K)\times\nb_{K,\sigma}\bigr]
\label{eq7.6}
\end{equation}
We re-write inequality \eqref{eq7.4} as
\begin{equation}
(\Pb+\Pb^T):[(\Sc+\Sc^T) (\Pb+\Pb^T)]\le0 \label{eq7.7}
\end{equation}

But due the assumption of vertex-symmetry, it is easy to compute that
$(\Sc+\Sc^T)$ has strictly positive eigenvalues and thus equation~\eqref{eq7.7} cannot
hold, and the contradiction is thus shown to be false. The opposite
case with $\mu_K=0$ is analogous, however here the contradiction follows
since for vertex-symmetric subcells $\sum_{\sigma\in\Fc_K\cap\Fc_s}m_\sigma^s
[(\aver{\xb}_{K,s}^\sigma-\xb_K )]\parallel\sum_{\sigma\in\Fc_K\cap\Fc_s}\nb_{K,\sigma}$. 
\end{proof}

\begin{corollary}\label{cor7.3}
For vertex-symmetric meshes, the local coercivity assumption~\ref{ass4.5} reduces to the inequality
\[
\abs{\ub}_{b_{\Dc,s}}^2\ge\theta'_{2,s} \sum_{\sigma\in\Fc_s}\frac{m_K^s}{d_{K,\sigma}^2}
\frac1{m_s^\sigma}\sum_{\beta\in\Gc_s^\sigma}\omega_\beta (\dblbr{\ub}_s^{\sigma,\beta})^2
\] 
for all $\ub\in\Pi_{FV,s} \Hbc_\Tc/\Rfr(\Omega)$, and some finite constants $\theta'_{2,s}$.
\end{corollary}

We remark that corollary~\ref{cor7.3} states that the coercivity assumption is
equivalent to verifying that the jumps in the elements of $\Pi_{FV,s}
\Hbc_\Tc$ are bounded by the mechanical energy. This is trivially verified for
hexagons, since the discretization is exact for deformations with constant
gradients. A straightforward calculation also verifies the property for
squares and parallelograms. However, we note that this property does
not hold on equilateral triangles when the full set of quadrature points
$\Gc_s^\sigma$ is used.

\subsection{Simplex grids}\label{sec7.2}

As a consequence of section~\ref{sec7.1}, we consider the reduced integration
proposed in Section~\ref{sec6.2} for simplex grids.

\begin{lemma}\label{lem7.4}
For simplex grids the local coercivity Assumption~\ref{ass4.5} with
the reduced integration of Section~\ref{sec6.2} holds.
\end{lemma}

\begin{proof}
Due to equation~\eqref{eq6.3} the discretization is symmetric, and the
local coercivity Assumption~\ref{ass4.5} simplifies to showing that the jump terms
are bounded, e.g.\ all $\ub\in\Pi_{FV,s} \Hbc_\Tc/\Rfr(\Omega)$
\begin{equation}
b_{\Dc,s} (\ub,\Pi_\Cc\ub)\ge\frac{\theta_{2,s}}{1-\theta_{2,s}}
\sum_{\sigma\in\Fc_s}\frac{m_K^s}{d_{K,\sigma}^2}\frac1{m_s^\sigma}
\sum_{\beta\in\Gc_s^\sigma}\omega_\beta (\dblbr{\ub}_s^{\sigma,\beta})^2 
\label{eq7.8}
\end{equation}
However, the reduced integration imposes strong continuity at the points
$\xb_\beta$ with $\beta=\Gc_\sigma^s$, and the right-hand side of equation~\eqref{eq7.8} is zero
for all $\ub\in\Pi_{FV,s} \Hbc_\Tc$. 
\end{proof}

The above shows that the local coercivity Assumption~\ref{ass4.5} always holds for
reduced integration on triangular grids. A similar situation was observed
for the scalar MPFA method in \cite{age10}. However, this does not imply that
the discretization is unconditionally convergent on triangular grids,
since as pointed out in Section~\ref{sec6.2}, the local problems \eqref{eq4.6}--\eqref{eq4.7} are
not always well-posed.

\subsection{Robustness with respect to Poisson ratio}\label{sec7.3}

It is of interest to understand the approximation quality of the method
with respect to the incompressible limit. This is known to be a challenge
for numerical methods, and is seen for e.g.\ lowest-order conforming finite
element methods on simplexes \cite{bre08}.

The standard approach to understanding the issue locking is to recognize
that in the limit $\lambda\to\infty$, the solution must satisfy the equations~\eqref{eq1.1}
with the parameter choice $\lambda=0$, subject to the constraint that
\begin{equation}
\nabla\cdot\ub=\Ob 
\label{eq7.9}
\end{equation}

This holds true both for the continuous and numerical solution. The
phenomenon of locking occurs when equation~\eqref{eq7.9} introduces more
constraints in the discrete system then the available degrees of freedom
\cite{nag74}.

For the methods discussed herein, we see that in the limit of $\lambda\to0$,
each of the local problems \eqref{eq4.9}--\eqref{eq4.11} are constrained to satisfy equation
\eqref{eq7.9}. Thus if the number of vertexes in $\Vc$, exceeds or equals the number
of degrees of freedom of the global system, given by $d$ times the number
of elements in $\Tc$, the finite volume discretization will lock. I.e., a
locking phenomena will arise if $\card(\Vc)>\card(\Tc)$. This is the case for
e.g.\ hexagons, where numerical locking has also been observed numerically
\cite{nor14b}.

To prove that locking will not occur for simplex grids, we need to
establish the existence of a Fortin operator for the method \cite{bre91}. In
essence, we need to establish the existence of an operator with the
following property.

\begin{definition}\label{def7.5}
We refer to an operator $\Pi_F\colon(H^1)^d\to\Pi_{FV,s}\Hbc_\Tc$
as a \emph{Fortin operator} if the following properties hold: 
For all $\ub\in(H_1)^d$ such that $\nabla\cdot\ub=\Ob$,
\begin{remunerate}
\item[\textup{a)}] $\bnabla_\Dc\cdot(\Pi_F\ub)=\Ob$
\item[\textup{b)}] There exists $c_F$ which does not scale as $h$ such that
\[
\abs{\Pi_\Tc \Pi_F \ub}_\Tc\le c_F \abs{\ub}_{H^1}
\]
\end{remunerate}
\end{definition}

The existence of a Fortin operator assures the robustness of the method
with respect to locking, and we recall the following lemma:

\begin{lemma}\label{lem7.6}
Let $\Pi_{FV,s}^\lambda$ explicitly represent the dependency of the
finite volume space on the parameter $\lambda$. Then if a Fortin operator exists,
it holds that
\[
\inf_{\wb\in\Pi_{FV,s}^\infty \Hbc_\Tc}\abs{\Pi_\Tc (\wb-\Pi_F\ub)}_\Tc \le c_\lambda
\inf_{\wb\in\Pi_{FV,s}^0 \Hbc_\Tc}\abs{\Pi_\Tc (\wb-\Pi_F\ub)}_\Tc 
\]
\end{lemma}

\begin{proof}
The proof uses classical arguments found in e.g.~\cite{bre91,han02,lem14}. 
\end{proof}

Lemma~\ref{lem7.6} guarantees that the approximation qualities of the space
$\Pi_{FV,s}^\lambda \Hbc_\Tc$ does not degenerate as $\lambda\to\infty$. Our task remains to show
the conditions under which a Fortin operator exists. To this end we will
use ideas from linear solvers, where static condensation and multiscale
finite volume methods provide the right partitioning of the grid.

The finite volume methods of the type discussed herein share the property
that the discretization stencil is local, wherein ``local'' implies
that the discrete conservation law for a cell $K\in\Tc$ depends only on cells
$L\in\Tc_s$ where $s\in\Vc_K$. This allows for a partitioning the tesselation $\Tc$
into a macromesh triplet $\Dfr= \{\Ifr,\Ffr,\Bfr\}$ such that each $K\in\Tc$ belongs to either
a macrocell $\mcell\in\Ifr$, a macroedge $\medge\in\Ffr$ or a makrovertex $\mvert\in\Bfr$. We refer to
e.g.~\cite{nor08} for the details of the macro-topology, but note the important
property that any two macrocells $\mcell,\bmcell\in\Ifr$ are completely separated from
the perspective of the discretization, such that static condensation
can be performed solely dependent on the unknowns in the macroedges and
macrovertexes. Furthermore, for all vertex $s\in\Vc$ it holds that at least
one $K\in\Tc_s$ lies in a macrocell $\mcell\in\Ifr$.

\begin{definition}\label{def7.7}
A family $\Dc_n$ of mesh triplets is \emph{locally underconstrained},
if for each $n$ there exists a macromesh triplet $\Dfr_n$ satisfying the following
properties:
\begin{remunerate} 
\item The maximum number of cells in any $\mcell\in\Ifr$ is bounded independent
of $n$.
\item For each $\mcell\in\Ifr$ it holds that $d\card(\Tc_\mcell)>\card(\Vc_\mcell)$.
\end{remunerate}
\end{definition}

Note that due to Definition~\ref{def7.7}a) the macrocell diameter scales
proportionally to the mesh diameter $h$. Regular family of meshes such as
triangulations satisfy Definition~\ref{def7.7}.

\begin{lemma}\label{lem7.8}
The discrete variational method \eqref{eq3.7}--\eqref{eq3.9}, and consequently
the MPSA O-method, is robust with respect to $\lambda\to\infty$ on families of
locally underconstrained meshes.
\end{lemma}

\begin{proof}
Let the operator $\Pi_F\colon(H^1)^d\to\Pi_{FV,s} \Hbc_\Tc$ be any operator
such that
\begin{equation}
(\Pi_F \ub)_K=\ub(\xb_K)\quad\text{for all }K\in\{\Ffr,\Bfr\}
\label{eq7.10}
\end{equation}
And such that
\begin{equation}
(\bnabla\cdot\ub)_K^s=0\quad\text{for all }K\in\Tc,\ s\in\Vc_K
\label{eq7.11}
\end{equation}

Due to Definition~\ref{def7.7}b) such operators exist, and due to Definition~\ref{def7.7}a) and scaling, any such operator defined consistently across the grid
sequence also satisfies Definition~\ref{def7.7}b). Thus $\Pi_F$ is a Fortin operator,
and the numerical method is robust. 
\end{proof}

Regular Cartesian lattices with Cartesian macrocell provides a limiting
case where $d\card(\Tc_\mcell)=\card(\Vc_\mcell)-1$ independent of the coarsening ratio,
and as such never satisfy Definition~\ref{def7.7}. Nevertheless, Cartesian grids
appear to be locking free based on numerical investigation \cite{nor14b}.

\section{Concluding remarks}\label{sec8}

We have expanded the hybrid finite volume framework and consider a
mixed discretization in the context of linear elasticity. We use ideas
from the variational multi-scale method to re-write the discretization
in a sub-space of cell-centered variables. By exploiting a discrete
Korn's inequality borrowed from the analysis of discontinuous Galerkin
methods we obtain global coercivity of the method, dependent only on a
locally computable coercivity condition. Convergence is then obtained
by compactness and consistency arguments. The discretization is designed
to be identical to the MPSA finite volume method recently proposed, and
convergence of that method is therefore established by the results herein.

The local coercivity conditions are simplified and verified a priori in
the setting of vertex-symmetric meshes in 2D and simplex meshes in 2D
and 3D. We furthermore establish necessary conditions for the robustness
of the scheme with respect to numerical locking using tools from mixed
methods and static condensation. Finally, we also identify a (new) finite
difference method based on the same framework, but for which no local
coercivity assumption is needed.

The analysis presented herein is fully consistent with the numerical
results presented in \cite{nor14b} on similar classes of grids, and provides a
rigorous understanding of the method.

The current work has not considered the convergence rate of the method,
and such results would necessarily require additional assumptions on
the regularity of the physical parameters $\mu$ and $\lambda$ and on the mesh
sequence $\Dc_n$.

\subsection*{Acknowledgements}
The author is currently associated with the Norwegian Academy of Science
and Letters through VISTA -- a basic research program funded by Statoil.

\bibliographystyle{siam}
\bibliography{refs}

\end{document}